\documentclass[a4paper,12pt]{article}
\usepackage[utf8]{inputenc}
\usepackage[T1]{fontenc}
\usepackage{amsmath}
\usepackage{amsthm}
\usepackage{amssymb}
\usepackage{enumerate}
\usepackage[all]{xy}
\usepackage{titlesec}

\usepackage{chngcntr}
\counterwithout{paragraph}{subsubsection}
\counterwithin{paragraph}{section}
\titlespacing*{\paragraph} {1em}{3.25ex plus 1ex minus .2ex}{0.5em}

\newtheoremstyle{mythms} 
{1em} 
{1em} 
{\itshape} 
{\parindent} 
{\bfseries} 
{} 
{.5em} 
{} 

\theoremstyle{mythms}
\newtheorem{theorem}[paragraph]{Theorem}
\newtheorem{proposition}[paragraph]{Proposition}
\newtheorem{lemma}[paragraph]{Lemma}
\newtheorem{corollary}[paragraph]{Corollary}

\renewenvironment{proof}{{\bfseries Proof.}}{\qed} 
 
 \newtheoremstyle{myrem} 
{1em} 
{1em} 
{\normalfont} 
{\parindent} 
{\bfseries} 
{} 
{.5em} 
{} 
 
\theoremstyle{myrem}
\newtheorem{remark}[paragraph]{Remark}

\newcommand{\Z}{\mathbb{Z}}

\newcommand{\G}{\mathbb{G}}
\newcommand{\F}{\mathbb{F}}

\newcommand{\EE}{\mathcal{E}}

\newcommand{\RR}{\mathcal{R}}

\renewcommand{\bf}[1]{\mathbf{#1}}

\newcommand{\id}{\mathrm{id}}

\newcommand{\coker}{\mathrm{coker}}
\newcommand{\Res}{\mathrm{Res}}
\newcommand{\Ind}{\mathrm{Ind}}
\newcommand{\Hom}{\mathrm{Hom}}

\newcommand{\Stab}{\mathrm{Stab}}
\newcommand{\Irr}{\mathrm{Irr}}
\newcommand{\IBr}{\mathrm{IBr}}
\newcommand{\ad}{\mathrm{ad}}
\newcommand{\Mod}{\mathrm{-mod}}
\newcommand{\normal}{\trianglelefteq}

\newcommand{\ol}[1]{\overline{#1}}
\newcommand{\wh}[1]{\widehat{#1}}
\newcommand{\wt}[1]{\widetilde{#1}}
\newcommand{\dw}{{\dot{w}}}

\title{On Extending Unipotent Representations to their Stabilizers}
\author{Matthias Klupsch}

\begin{document}

\maketitle
\begin{abstract}
 We show that irreducible unipotent  representations of split Levi subgroups of finite groups
 of Lie type extend to their stabilizers inside the normalizer of the given Levi subgroup. For this purpose, we extend the multiplicity-freeness theorem from regular embeddings to arbitrary isotypies. 
\end{abstract}

\section{Introduction}
 Given a normal subgroup $N$ of a finite group $G$, one obtains 
 an action of the quotient group $G/N$ on the set of irreducible
 representations of $N$. If an irreducible representation of $N$
 is fixed under this action, it is straightforward to ask if 
 this representation might extend to a representation of $G$.
 
 That this is not always the case is nicely explained by Schur's
 theory of projective representations. 
 The obstruction for an irreducible representation of $N$ to extend
 to a representation of $G$ lies in (the cohomology class of) a certain 
 $2$-cocycle associated with this representation.
 
 When considering representations of finite groups of Lie type,
 the question of whether irreducible cuspidal representations of split Levi subgroups
 extend to their stabilizers inside the normalizer of the given Levi subgroup
 is of much interest (cf. \cite[Cor. 3.13]{geck1996towards}).
 
In the case of representations in characteristic $0$, this question was answered by Lusztig and Geck to the affirmative (see \cite[Thm. 8.6.]{lusztig1984characters} and \cite{geck1993harish}). 
 
 Our goal with this paper is to better understand which irreducible representations of Levi subgroups extend to their stabilizers. 
 Among the irreducible representations of finite groups of Lie type, we have singled out the unipotent ones (in arbitrary non-defining characteristic) to serve as a first working example.
 
 In fact, we shall show that every irreducible unipotent representation of a split Levi subgroup of a finite group of Lie type extends to its stabilizer.
 
 In order to reduce this extension problem to a situation we can handle easily
 we need to generalize some results on unipotent representations especially
 with regards to their behaviour concerning restriction along isotypies.
 In particular, we show that restriction along isotypies maps arbitrary irreducible 
 representations to multiplicity-free representations (known previously only for regular embeddings, cf. \cite[1.7.15]{geck2020character}).
 
 This way, we are able to reduce the extension problem for unipotent representations
 to a problem which only concerns semidirect products with wreath products
 which is in some sense a  generalization of \cite[Prop. 4.3]{geck1996towards}.
 
 The structure of this paper is as follows. First we describe the purely representation theoretic problem of extending an irreducible representation
 to a semidirect product of a group with a wreath product. 
 
 Then we shall show that the normalizers of Levi subgroups are of this form
 modulo a central subgroup. 
 
 Afterwards, we prove that restriction along isotypies of irreducible representations 
 leads to multiplicity-free representations reducing this problem to the already known multiplicity-freeness theorem \cite[1.7.15]{geck2020character} for regular embeddings. Specializing this to unipotent 
 representations, we find that the restriction of an irreducible unipotent 
 representation along an isotypy remains irreducible. This was already well-known
 in characteristic $0$ (see \cite[2.3.14]{geck2020character}).

 Finally, we will be able to solve the extension problem for irreducible unipotent representation.
 
\section{Some Clifford theory}

\paragraph{}
In this section, we will consider the property of extending
representations to their stabilizer and establish connections 
to multiplicity-freeness of restriction functors.
 
\paragraph{} Let $G$ be a finite group, $N \normal G$ a normal subgroup of $G$ and $k$ an algebraically closed field of characteristic $\ell \geq 0$. 
We shall denote by $\Irr_k(G)$ the set of irreducible $k$-characters of $G$.

\paragraph{}
The $k$-linear categories $k N$-mod and $k G$-mod of finitely generated $k N$-modules and 
 $k G$-modules are related via 
 induction and restriction functors 
 $$ \Ind^G_N : k N\Mod \to kG\Mod$$
 and 
 $$ \Res^G_N: kG\Mod\to kN\Mod$$
 which are biadjoint exact functors.
 Recall that Clifford theory tells us that $\Res^G_N$ maps semisimple modules to semisimple modules.
 
 \paragraph{}
 For any $g \in G$, conjugation with $g$ defines an automorphism of $N$ and thus 
 an autoequivalence of categories 
 $$ \ad(g): kN\Mod \to kN\Mod.$$
 If $M$ is a $kN$-module, then $n \in N$ acts on $\ad(g)M$ like $g^{-1} n g$ acts on $M$.
 We denote by 
 $$S_G(M) = \Stab_G(M) = \{g \in G \:|\: \ad(g) M \cong M\}$$
 the stabilizer of $M$ in $G$.

 \begin{lemma}\label{la: equivalence multiplicity-free and extension}
  Let $X_0$ be a simple $k N$-module. The following
  statements are equivalent:
  \begin{enumerate}[(i)]
   \item There exists a simple $kG$-module $X$ for which
         $\Res^G_N X$ is multiplicity-free and has $X_0$ as a direct summand.
   \item There exists a simple $k S_G(X_0)$-module 
         $X_1$ such that $\Res^{S_G(X_0)}_{N} X_1 \cong X_0$. 
  \end{enumerate}
   In fact, if $X$ is as in (i), then any simple submodule of $\Res^{G}_{S_{G}(X_0)} X$ whose restriction to $N$ has $X_0$ as a direct summand satisfies (ii) and 
   any $k S_G(X_0)$-module $X_1$ as in (ii) satisfies that $\Ind^G_{S_G(X_0)} X_1$ is simple and satisfies (i).
 \end{lemma}
 \begin{proof}
  Suppose first, that (i) holds and let $X_1$ be a simple $k S_{G}(X_0)$-submodule of $\Res^{G}_{S_{G}(X_0)} X$ for which $X_0$ is a direct summand of $\Res^{S_G(X_0)}_N X_1$.
  
By usual Clifford theory, $\Res^{S_G(X_0)}_N X_1$ is a direct sum of simple $k N$-modules which are conjugate to $X_0$ under $S_{G}(X_0)$, hence isomorphic to $X_0$. 
By (i), however, $\Res^{S_G(X_0)}_N X_1$ is multiplicity-free as a direct summand of a multiplicity-free module. We thus obtain (ii).

Let us assume (ii) now and let $X = \Ind^G_{S_G(X_0)} X_1$. 
Using the Mackey formula, we see that 
$$\Res^{G}_N X \cong \bigoplus_{g \in S_G(X_0) \backslash G / S_G(X_0)} \ad(g)X_0 $$
which is clearly multiplicity-free.

If $T$ is a simple submodule of $X$, then $\Res^{G}_{N} T$ is a direct sum of conjugates of $X_0$ with each conjugate occurring at least once. Since $\Res^G_{N} X$ is multiplicity free, it follows that we have $\Res^G_{N} X = \Res^G_{N} T$ and so $X = T$
is simple.
%
%
%
 \end{proof}

\paragraph{} A very prominent example for multiplicity-freeness
in the context of finite groups of Lie type is the well-known
result that restriction along regular embeddings is multiplicity-free (see \cite[1.7.15]{geck2020character}). We shall later return to this result and generalize it to arbitrary isotypies. For the moment, we just note that in this situation, the relevant quotient group $G/N$ is abelian which makes this special case very important for us. The following result will later be used 
to show that unipotent representations in positive characteristic
remain irreducible under restriction along isotypies.

 \begin{lemma}\label{la: abelian multiplicity-free}
  Suppose that $G/N$ is abelian and that $X$ is a simple $kG$-module.
  The following statements are equivalent:
  \begin{enumerate}[(i)]
   \item The $k N$-module $\Res^G_N X$ is  simple.
   \item The $k N$-module $\Res^G_N X$ is multiplicity-free and for every nontrivial $\lambda \in \Irr_k(G/N)$ we have $\lambda \otimes X \not\cong X$ 
  \end{enumerate}
\end{lemma}
\begin{proof}
 Suppose (i) holds and that $f: \lambda \otimes X \to X$ is an isomorphism of $k G$-modules for some $\lambda \in \Irr_k(G/N)$.
 
 It follows that $\Res^G_N f : \Res^G_N X \to \Res^G_N X$ is an 
 automorphism of a simple $k N$-module and Schur's Lemma now implies
 that $\Res^G_N f = \alpha \id_{\Res^G_N X}$ for some $\alpha \in k^*$.
 
 It follows that we have 
 $$\alpha \lambda(g) gx = f(\lambda(g) g x) = g f(x) = \alpha gx $$
 for all $g \in G$ and $x \in X$ and so $\lambda$ is trivial.
 
 Conversely, let us now assume (ii) and let $X_0$ be a simple direct summand of $\Res^G_N X$ and $S_G(X_0) \subseteq G$ its stabilizer.
 Also let $X_1$ be a simple submodule of $\Res^G_{S_G(X_0)} X$ such that $\Res^{S_G(X_0)}_N X_1 \cong X_0$, which exists
 by Lemma \ref{la: equivalence multiplicity-free and extension} since we assumed 
 $\Res^G_N X$ to be multiplicity-free.
 
 Moreover, we have $X \cong \Ind^{G}_{S_G(X_0)} X_1$ by the same Lemma.  
 Now, since $G/N$ is abelian, the restriction $\Res^{G}_{S_G(X_0)} X$ is a direct sum of conjugates of $X_1$ and thus 
 $$ \Ind^G_{S_G(X_0)} \Res^G_{S_G(X_0)} X \cong k G/S_G(X_0) \otimes X $$
 is semisimple which implies that the order of $G/S_G(X_0)$ is prime to $\ell$.
 
 Now, if $S_G(X_0)$ were a proper subgroup of $G$, then there would exist a nontrivial $\lambda \in \Irr_k(G/S_G(X_0))$ and we would obtain  
 $$ \lambda \otimes X \cong \Ind^G_{S_G(X_0)} (\Res^G_{S_{G}(X_0)}(\lambda) \otimes X_1) \cong X $$
 contrary to our assumptions. Thus, we have $G = S_G(X_0)$ and $X \cong X_1$ which gives us (i). 
\end{proof}

\begin{remark}
Suppose that $G = N \rtimes W$ for some subgroup $W \subseteq G$.
Note that the centralizer $C = C_{W}(N) \subseteq W$ is a normal subgroup of $G$. We let $G' = N \rtimes (W/C) \cong G/C$.

 In this situation, if $X_0$ is a simple $k N$-module, then we have $C \subseteq S_G(X_0)$ and $S_{G'}(X_0) = S_G(X_0)/C$.
 Moreover, if there exists a simple $k S_G'(X_0)$-module $X'$
 such that $\Res^{S_{G'}(X_0)}_N X' \cong X_0$, then the inflation $X$
 of $X'$ to $G$ satisfies $\Res^{S_G(X_0)}_N X \cong X_0$.
\end{remark}

\paragraph{}
In the next section, we shall investigate the structure of the normalizer of a standard Levi subgroup of a finite group of Lie 
type. We shall see that the following result on extending 
representations for certain semidirect products is exactly 
what we need.

  \begin{theorem}\label{thm: extending to wreath products}
 Let $H$ and $A$ be finite groups with $A$ acting on $H$ via group automorphisms. For a positive integer $n \in \Z_{> 0}$ we set 
 $G = (H \rtimes A) \wr S_n = H^n \rtimes (A \wr S_n)$
 as well as $N = H^n \normal G$.
 
 If every simple $k H$-module $X_0$ extends to a 
 $k S_{H \rtimes A}(X_0)$-module, 
 then every simple $k N$-module $X$ extends to a $k S_G(X)$-module.
\end{theorem}
\begin{proof}
 Let $X$ be a simple $kN$-module. 
 We can write 
 $$  X = \bigotimes_{i = 1}^n X_{i}$$
 where $X_i$ is a simple $k H$-module.

 We identify $H^n$ with $H^{\{1,\dots,n\}}$ and for a subset $J \subseteq \{1,\dots,n\}$ we shall denote by $H^{J}$ the subgroup
 of $H^n$ which under the $i$th projection $H^n \to H$ has image $H$ if $i \in J$ and $\{1\}$ else.
 
 We have the equivalence relation
 $i \sim j$ if and only if $X_i$ is isomorphic to $\ad(a) X_j$ for some $a \in A$ and write $\{1,\dots,n\} = \bigsqcup_{j = 1}^r I_j$ for the associated partition. The stabilizer $S_G(X)$ is a subgroup of 
 $$ \prod_{j = 1}^r \left( H^{I_j} \rtimes (A \wr S_{I_j}) \right)$$
 where we denote the $i$th factor by $G_{I_j}$.
 
 Let us set $X_{I_j} = \bigotimes_{i \in I_j} X_i$ so that we have
 $X = \bigotimes_{j = 1}^r X_{I_j}$.
 It follows that we have $S_{G}(X) = \prod_{j = 1}^r S_{G_{I_j}}(X_{I_j})$.
 Clearly, there exists a $k S_{G}(X)$-module $Y$ with
 $\Res^{S_G(X)}_{N} Y \cong X$ if only if for every 
 $1 \leq i \leq r$ we have 
 $$\Res^{S_{G_{I_j}}(X_{I_j})}_{N} Y_{I_j} \cong X_{I_j}$$ 
 for some $k S_{G_{I_j}}(X_{I_j})$-module $Y_{I_j}$.
 
 We may thus assume that $r = 1$.
 For each $2 \leq i \leq n$ let $a_i \in A$ be such that 
 $\ad(a_i)X_i \cong X_1$. With $a = (a_i)_i \in A^n$,
 we get $\ad(a) X \cong X_1^{\otimes n}$.
 
 Clearly, if $\ad(a) X$ extends to a $k S_G(\ad(a) X)$-module,
 then $X$ extends to a $k S_G(X)$-module.
 We may thus assume that $X \cong X_1^{\otimes n}$.
 Note that we then have $S_G(X) = S_{H \rtimes A}(X_1) \wr S_n$.
 
 Let $\wt{X_1}$ be an extension of $X_1$ to $S_{H \rtimes A}(X_1)$. The 
 $k (H \rtimes A)^n$-module $\wt{X} = \wt{X_1}^{\otimes n}$ becomes 
 a $k S_G(X)$-module in a natural way and we are done. 
\end{proof}
 
\section{Representatives for the relative Weyl group}

\subsection{Root data and isotypies}

\paragraph{}
Recall that a root datum is a quadruple $\RR = (X,\Phi,X^\vee,\Phi^\vee)$ where $X$ and $X^\vee$ are free abelian groups of finite rank together with a perfect pairing $X \times X^\vee \to \Z$ 
$\Phi$ and $\Phi^\vee$ are finite subsets of $X$ and $X^\vee$, respectively, together with a bijection $\Phi \to \Phi^\vee, \alpha \mapsto \alpha^\vee$
such that 
\begin{enumerate}[(i)]
 \item $\langle \alpha , \alpha^\vee \rangle = 2$ for all $\alpha \in \Phi$,
 \item $2 \alpha \notin \Phi$ for all $\alpha \in \Phi$,
 \item $\Phi$ and $\Phi^\vee$ are invariant under the simple reflections $s_\alpha : X \to X$ and $s_\alpha^\vee: X^\vee \to X^\vee$ defined 
by 
$$ s_\alpha(\chi) = \chi - \langle \chi, \alpha^\vee \rangle \alpha$$
and 
$$ s_\alpha^\vee(\gamma) = \gamma - \langle \alpha, \gamma \rangle \alpha^\vee,$$
respectively.
\end{enumerate}
Given a set $\Delta \subseteq \Phi$ of simple roots, the quintuple
$\wh{\RR}= (X,\Phi,X^\vee,\Phi^\vee,\Delta)$ is called a \emph{based root datum} associated with $\RR$.

\paragraph{}
Let $\RR = (X,\Phi,X^\vee,\Phi^\vee)$ and $\RR' = (X',\Phi',X'^\vee,\Phi'^\vee)$ be two root data and $p$ a prime number.
Like in \cite{taylor2019rootdata}, a $p$-morphism of root 
data $\RR' \to \RR$ is a triple $(f,q,\tau)$ where $f : X' \to X$
is a morphism of groups, $q : \Phi \to p^{\Z_{\geq 0}}$ a map and 
$\tau : \Phi \to \Phi'$ a bijection such that 
$$f(\tau(\alpha)) = q(\alpha) \alpha  $$
for all $\alpha \in \Phi$ and 
$$f^\vee(\alpha^\vee) = q(\alpha) \tau(\alpha)^\vee $$
where $f^\vee: X'^\vee \to X^\vee$ is the transposed map of $f$
with respect to the perfect pairings between $X$ and $X^\vee$ 
as well as $X'$ and $X'^\vee$. 
As the map $f$ already determines $q$ and $\tau$, we shall also sometimes call $f$ a $p$-morphism of root data.
\paragraph{}
If $\Delta \subseteq \Phi$ and $\Delta'\subseteq \Phi'$ are sets 
of simple roots such that $\tau(\Delta) = \Delta'$, then we shall say that  
$(f,q,\tau)$ (or simply $f$) is a $p$-morphism of based root data 
$\wh{\RR} \to \wh{\RR'}$. 
\paragraph{}
We shall write $\Hom_{p}(\RR',\RR)$ for the set of $p$-morphisms of root data $\RR' \to \RR$ and $\Hom_{p}(\wh{\RR'},\wh{\RR})$ for the set of $p$-morphisms of based root data $\wh{\RR'} \to \wh{\RR}$.

\paragraph{} In the following, all our algebraic groups will be defined over $\F = \ol{\F_p}$, the algebraic closure of the prime field $\F_p$ with $p$ elements.
The algebraic groups $\G_a$ and $\G_m$ are just the additive group $(\F,+)$ and the multiplicative group $(\F - \{0\}, \cdot)$, respectively. 

\paragraph{} 
Let $\bf{G}$ be a connected reductive group and $\bf{T} \subseteq \bf{G}$ a maximal torus of $\bf{G}$.
The root datum of $\bf{G}$ associated to $\bf{T}$
will be denoted by $$\RR(\bf{G},\bf{T}) = (X(\bf{T}),\Phi(\bf{G},\bf{T}),X^\vee(\bf{T}),\Phi^\vee(\bf{G},\bf{T})) = (X,\Phi,X^\vee,\Phi^\vee).$$
Here, $X(\bf{T})$ is the character group of $\bf{T}$ and $X^\vee(\bf{T})$ the cocharacter group.

If $\bf{B} \subseteq \bf{G}$ is a Borel subgroup containing $\bf{T}$ and $\Delta \subseteq \Phi$ the corresponding set of simple roots, we let $\wh{\RR}(\bf{G},\bf{B},\bf{T})$ denote the corresponding based root system. 

\paragraph{}
For any element $g \in \bf{G}$ conjugation with $g$ will be denoted by $c_g : \bf{G} \to \bf{G}$.

\paragraph{}
For every $\alpha \in \Phi$, there exists a morphism 
$u_\alpha : \G_a \to \bf{U}_\alpha$ where $\bf{U}_\alpha \subseteq \bf{G}$ is a minimal connected unipotent subgroup of $\bf{G}$
normalized by $\bf{T}$ such that 
$$ t u_\alpha(x) t^{-1} = u_\alpha(\alpha(t)x) $$
for all $x \in \G_a$ and $t \in \bf{T}$. 
For $\bf{B} \subseteq \bf{G}$ a Borel subgroup containing $\bf{T}$ and $\Delta \subseteq \Phi$ its corresponding set of simple roots fixing morphisms $\{u_\alpha\}_{\alpha \in \Delta}$ as above,  
we shall call the tuple $(\bf{G},\bf{B},\bf{T}, \{u_\alpha \}_{\alpha \in \Delta})$
a \emph{pinning} of $(\bf{G},\bf{T})$.

\paragraph{} 
Let $\bf{G}'$ be another connected reductive group and $\varphi: \bf{G} \to \bf{G}'$ a morphism 
of algebraic groups. Recall that $\varphi$ is called an \emph{isotypy}
if $\varphi(\bf{G})$ contains the derived subgroup $[\bf{G}',\bf{G}']$ and $\ker(\varphi)$ is central.
Since we have $\bf{G}' = Z(\bf{G}')[\bf{G}',\bf{G}']$, it follows that 
one has $\bf{G}' = Z(\bf{G}')\varphi(\bf{G})$ in this case.

\paragraph{}\label{par: isotypy maps to p-morphism}
Moreover, if $\bf{T} \subseteq \bf{G}$ and $\bf{T}'\subseteq \bf{G}'$
are maximal tori such that $\varphi(\bf{T}) \subseteq \bf{T}'$,
then we shall write $\varphi : (\bf{G},\bf{T}) \to (\bf{G}',\bf{T}')$. If $\varphi$ is an isotypy $(\bf{G},\bf{T}) \to (\bf{G}',\bf{T}')$, then we obtain a $p$-morphism of root data
$$\RR(\varphi) = (f_\varphi,q,\tau): \RR(\bf{G}',\bf{T}') \to \RR(\bf{G},\bf{T})$$ where $f_\varphi(\chi') = \chi' \circ \varphi$ for $\chi' \in X(\bf{T}')$.
Clearly, if $\psi': (\bf{G}',\bf{T}') \to (\bf{G}'',\bf{T}'')$
is a second isotypy, then we have $\RR(\psi \circ \varphi) = \RR(\varphi) \circ \RR(\psi)$.

\paragraph{}\label{par: isotypies compatible with pinnings}
Furthermore, if $(\bf{G},\bf{B},\bf{T},\{u_\alpha\}_{\alpha \in \Delta})$ and $(\bf{G}',\bf{B}',\bf{T}',\{u_{\alpha'}\}_{\alpha' \in \Delta'})$ are pinnings of $(\bf{G},\bf{T})$ and $(\bf{G}',\bf{T}')$, respectively, we shall say that $\varphi$ is compatible with these pinnings and write 
$\varphi : (\bf{G},\bf{B},\bf{T},\{u_\alpha\}) \to (\bf{G}',\bf{B}',\bf{T}',\{u_{\alpha'}\})$ if we have $\varphi(\bf{B}) \subseteq \bf{B}'$ and $\varphi(\bf{T}) \subseteq \bf{T}'$ 
as well as 
$$ \varphi(u_{\alpha}(x)) = u_{\tau(\alpha)}(x^{q(\alpha)}) $$
for all $x \in \G_a$ and $\alpha \in \Delta$.

\paragraph{}
We write $\Hom((\bf{G},\bf{T}),(\bf{G}',\bf{T}'))$
for the set of isotypies $(\bf{G},\bf{T}) \to (\bf{G}',\bf{T}')$
and likewise 
$\Hom((\bf{G},\bf{B},\bf{T},\{u_\alpha\}_{\alpha \in \Delta}),(\bf{G}',\bf{B}',\bf{T}',\{u_\alpha'\}_{\alpha' \in \Delta'}))$
for the set of isotypies described
in \ref{par: isotypies compatible with pinnings}
\begin{proposition}[{\cite[Thm. 3.8., Prop. 3.13]{taylor2019rootdata}}]\label{prop: bijection isotypies p-morphisms}
 The map
 \begin{align*}
  \Hom((\bf{G},\bf{T}),(\bf{G}',\bf{T}')) &\to \Hom_{p}(\RR(\bf{G}',\bf{T}'),\RR(\bf{G},\bf{T})),
 \\ \varphi &\to \RR(\varphi)
 \end{align*}
is surjective and $\RR(\varphi_1) = \RR(\varphi_2)$ if and only if 
$\varphi_2 = c_{t'} \circ \varphi_1$ for some $t' \in \bf{T}'$.
The map
\begin{align*}
 \Hom((\bf{G},\bf{B},\bf{T},\{u_\alpha\}_{\alpha \in \Delta}),(\bf{G}',\bf{B}',\bf{T}',&\{u_\alpha'\}_{\alpha' \in \Delta'})) 
\to \\ &\Hom_{p}(\wh{\RR}(\bf{G}',\bf{T}'),\wh{\RR}(\bf{G},\bf{T})), 
\\ \varphi \to &\RR(\varphi) 
\end{align*}
is a bijection.
\end{proposition}

\begin{remark}\label{remark: isotypies and pinnings}
 Note that a consequence of this is that if we have an isotypy
 $\varphi: (\bf{G},\bf{T}) \to  (\bf{G}',\bf{T}')$  with $\varphi(\bf{B}) \subseteq \bf{B}'$,  
 then there exists $t' \in \bf{T}'$ such that 
 $$ \varphi(u_{\alpha}(x)) = t' u_{\tau(\alpha)}(x^{q(\alpha)}) t'^{-1}  = u_{\tau(\alpha)}(\tau(\alpha)(t') x^{q(\alpha)}) $$
 for all $x \in \G_a$ and $\alpha \in \Delta$.
 In other words, the $\bf{T}'$-conjugate $c_{t'^{-1}} \circ \varphi$ is an isotypy 
 $(\bf{G},\bf{B},\bf{T},\{u_\alpha\}) \to (\bf{G}',\bf{B}',\bf{T}',\{u_{\alpha'}\})$.
\end{remark}

\begin{remark}\label{remark: composition of compatible isotypies}
 For $i = 1,2$, let $\varphi_i: (\bf{G}_{i},\bf{T}_{i}) \to (\bf{G}_{i+1},\bf{T}_{i+1})$ be an isotypy between connected reductive groups and set $\varphi_3 = \varphi_2 \circ \varphi_1$.
 Suppose that we have pinnings $(\bf{G}_{i},\bf{B}_{i},\bf{T}_{i},\{u_\alpha\}_{\alpha \in \Delta_i})$ of $(\bf{G}_i,\bf{T}_i)$ ($i = 1,2,3$).
 Then if two of the three morphisms $\varphi_i$ are  
 compatible with the given pinnings, it follows easily that the third one is compatible as well. 
\end{remark}

\subsection{Steinberg morphisms and isotypies}
\paragraph{}
Recall from \cite[1.4.1 -- 1.4.5]{geck2020character} that a Frobenius morphism of $\bf{G}$ is 
a morphism $F : \bf{G} \to \bf{G}$ of algebraic groups that encapsulates an 
$\F_q$-rational structure on $\bf{G}$ for some $p$-power $q$ such that $\F_{q}$-rational points are the set of fixed points $\bf{G}^F$.
More generally, a Steinberg morphism of $\bf{G}$ is a 
morphism $F : \bf{G}\to\bf{G}$ of algebraic groups 
for which some power $F^m$ is a Frobenius morphism.
Steinberg morphisms are bijections and thus, in particular,
isotypies.

\paragraph{}
We say that a pinning 
$(\bf{G},\bf{B},\bf{T},\{u_\alpha\}_{\alpha \in \Delta})$
is compatible with a Steinberg morphism $F : \bf{G} \to \bf{G}$
if $F$ is an isotypy 
$(\bf{G},\bf{B},\bf{T},\{u_\alpha\}_{\alpha \in \Delta}) \to 
(\bf{G},\bf{B},\bf{T},\{u_\alpha\}_{\alpha \in \Delta})$.
Note that $\bf{T}$ is then a maximally split torus
in the sense of \cite[1.4.10]{geck2020character}.

\begin{lemma}
 Let $F : \bf{G} \to \bf{G}$ be a Steinberg morphism.
 Then there exists an $F$-stable maximal torus $\bf{T} \subseteq \bf{G}$ and a pinning 
 $(\bf{G},\bf{B},\bf{T},\{u_\alpha\}_{\alpha \in \Delta})$ of $(\bf{G},\bf{T})$
which is compatible with $F$. 
\end{lemma}
\begin{proof}
By \cite[1.4.10]{geck2020character} there exists an $F$-stable maximal torus $\bf{T} \subseteq \bf{G}$ and an $F$-stable 
Borel subgroup $\bf{B}$ containing $\bf{T}$.
Let $(\bf{G},\bf{B},\bf{T},\{u_\alpha\}_{\alpha \in \Delta})$ be some pinning. It follows from Remark (\ref{remark: isotypies and pinnings})  that there exists $t' \in \bf{T}$ such that we have 
$$ F(u_{\alpha}(x)) = t' u_{\tau(\alpha)}(x^{q(\alpha)}) t'^{-1}$$
for all $x \in \G_a$ and $\alpha \in \Delta$. Here, $\tau$ and $q$ are like in 
\ref{par: isotypy maps to p-morphism} associated with the isotypy $F$.
By \cite[1.4.8]{geck2020character}, there exists $t \in \bf{T}$ such 
that $F(t) = t t'^{-1}$ and thus 
$$ F(t u_{\alpha}(x) t^{-1}) = t u_{\tau(\alpha)}(x^{q(\alpha)}) t^{-1}$$
and so the pinning 
$(\bf{G},\bf{B},\bf{T},\{c_t \circ u_\alpha \}_{\alpha \in \Delta})$ is compatible with $F$.
\end{proof}

\begin{lemma}\label{la: rational isotypies and pinnings}
 Let $(\bf{G},\bf{B},\bf{T},\{u_\alpha\}_{\alpha \in \Delta})$ 
 and $(\bf{G}',\bf{B}',\bf{T}',\{u_{\alpha'}\}_{\alpha' \in \Delta'})$
 be two pinnings and let $F :\bf{G} \to \bf{G}'$ and $F': \bf{G}' \to \bf{G}'$ be Steinberg
 morphisms which are compatible with these.
 
 If $\varphi: (\bf{G},\bf{T}) \to (\bf{G}',\bf{T}')$ is an isotypy with $\varphi(\bf{B}) \subseteq \bf{B}'$ and $F' \circ \varphi = \varphi \circ F$, then for every
  $t' \in \bf{T}'$ such that $\psi = c_{t'} \circ \varphi $ is compatible
  with the given pinnings
 we have $F' \circ \psi = \psi \circ F$.
 
\end{lemma}
\begin{proof}
 Let $t' \in \bf{T}'$ be 
 such that 
 $\psi = c_{t'} \circ \varphi$ 
 is 
 compatible with the pinnings.
 
 Since $F' \circ \varphi = \varphi \circ F$ and $F$ and $F'$ are compatible with the pinnings, Proposition
 \ref{prop: bijection isotypies p-morphisms} implies that $F' \circ \psi$ and $\psi \circ F$ are two isotypies which are compatible
 with the given pinnings and which induce the same $p$-morphism of root data. Proposition
 \ref{prop: bijection isotypies p-morphisms}
 now implies that these isotypies coincide.
\end{proof} 
  
\begin{remark}\label{remark: rational isotypies and pinnings}
 Note that \ref{remark: isotypies and pinnings} implies that 
 there always exist a $t' \in \bf{T}'$ which satisfies 
 the hypothesis of Lemma \ref{la: rational isotypies and pinnings}
 
 Proposition \ref{prop: bijection isotypies p-morphisms} implies that the isotypy $\psi = c_{t'} \circ \varphi$ is the unique $\bf{T}'$-conjugate of $\varphi$ 
 which which is compatible with the pinnings. It follows that the element $t' \in \bf{T}'$ is unique up to
 an element of $Z(\bf{G}')$.

 Furthermore, the identities
 $F' \circ \varphi = \varphi \circ F$ and 
 $F' \circ \psi = \psi \circ F$ imply that 
 $$c_{t'^{-1} F'(t')} F' \circ \varphi = \varphi \circ F = F' \circ \varphi $$
 and since we have $\bf{G}' = Z(\bf{G}')^\circ \varphi(\bf{G})$ 
 it follows that $t'^{-1}F'(t') \in Z(\bf{G}')$, so the 
 coset $t' Z(\bf{G}')$ is $F'$-stable. In particular,
 if $Z(\bf{G}')$ is connected, then \cite[1.4.8]{geck2020character} implies that $t' \in \bf{T}'$
 can be chosen to satisfy $F'(t') = t'$.
\end{remark}

\subsection{Normalizers of standard Levi subgroups}



\paragraph{}
As before, we fix a connected reductive group $\bf{G}$
with maximal torus $\bf{T}$ and assume that we have a Steinberg morphism $F : \bf{G} \to \bf{G}$ and a pinning $(\bf{G},\bf{B},\bf{T},\{u_\alpha\}_{\alpha \in \Delta})$ which is compatible with $F$.
We denote the corresponding based root system by $\wh{\RR} = (X,\Phi,X^\vee,\Phi^\vee,\Delta)$ and its Weyl group $W$ will be identified with $N_{\bf{G}}(\bf{T})/\bf{T}$.

\paragraph{}
For any closed $F$-stable subgroup $\bf{H} \subseteq \bf{G}$
we shall denote the group of $F$-fixed points $\bf{H}^{F}$
by the corresponding non-boldfaced letter $H$.
\paragraph{}\label{par: Steinberg Weyl group fixed points}
The Steinberg morphism $F$ induces an automorphism of $W$ which will also be denoted by $F$.
The group $W^F$ of $F$-fixed points of $W$ is a finite Coxeter group whose simple reflections are the longest elements $w_J$ of standard parabolic subgroups $W_J$ where $J \subseteq \Delta$ is a $\tau$-orbit (cf. \cite[Theorem 1]{geck2014coxeter}).

\paragraph{}\label{par: Steinberg and Weyl group action}
Furthermore, the Steinberg morphism $F$ induces a $p$-endomorphism of based root data $(F,q,\tau)$ on $\wh{\RR}$ and we have  
$$\tau(w \alpha) = F(w) \tau(\alpha)$$
for all $w \in W$ and $\alpha \in \Phi$.

\paragraph{}
For any subset $I \subseteq \Delta$ we have a standard parabolic subgroup $\bf{P}_I$ with Levi decomposition
$\bf{P}_{I} = \bf{L}_I \bf{U}_I$ and all three groups are $F$-stable if $I$ is $\tau$-stable. The standard Levi subgroup $\bf{L}_I$ is a connected reductive group with root datum $\RR_I = (X,\Phi_I,X^\vee,\Phi_I^\vee)$ where $\Phi_I = \Phi \cap \Z I$.

\paragraph{}
We shall identify the Weyl group of $\bf{L}_{I}$ with the standard parabolic subgroup $W_I$ of $W$ generated by the reflections associated with the simple roots of $I$.
It follows from \ref{par: Steinberg Weyl group fixed points} that if $I \subseteq \Delta$ is $\tau$-stable, then $W_I^F$ is a standard parabolic subgroup of $W^F$.

\paragraph{}
The Levi subgroup $\bf{L}_I$ has finite index in its normalizer $N_{\bf{G}}(\bf{L}_{I})$ and one has an isomorphism
$$N_{W}(W_I)/W_I \cong N_{\bf{G}}(\bf{L}_{I})/\bf{L}_{I} $$
as every representative $\dw \in N_{\bf{G}}(\bf{T})$ of an element $w \in N_{W}(W_I)$ is an element of $N_{\bf{G}}(\bf{L}_I)$.

\paragraph{}
Moreover, \cite[Cor. 3]{howlett1980normalizers} tells us that the group $N_{W}(W_I)$ is a semidirect product 
$$N_{W}(W_I) = W_I \rtimes N_{W}(I)$$
where $N_{W}(I) = \{w \in W \:|\: w(I) = I\}$. Suppose now that $I$ is $\tau$-stable. 
Taking $F$-fixed points, we obtain 
$$N_{W^F}(W_I) = {W_I}^F \rtimes N_{W^F}(I)$$
where $N_{W^F}(I) = \{w \in W^F \:|\: w(I) = I\} = W^F \cap N_{W}(I)$
and so 
$$N_{W^F}(W_I)/{W_I}^F \cong N_{G}(\bf{L}_{I})/{L}_{I}.$$
Recall that we write $G$ for the $F$-fixed points $\bf{G}^F$ and similarly $L_I$ for $\bf{L}_I$.

\paragraph{}
On the other hand, $W_I^F$ is a standard parabolic subgroup 
of $W^F$ and thus \cite[Cor. 3]{howlett1980normalizers} implies that $N_{W^F}(W_I^F)$ is a semidirect product 
$$N_{W^F}(W_I^F) = W_I^F \rtimes N_{W^F}(I,\tau)$$
where, in view of \ref{par: Steinberg Weyl group fixed points},  the group 
$N_{W^F}(I,\tau)$ is defined to consist of those $w \in W^F$ for which for every $\tau$-orbit $J \subseteq I$ we have $w w_J w^{-1} = w_{J'}$ for some $\tau$-orbit $J' \subseteq I$.

\begin{lemma}
 We have $$N_{W^F}(I) = N_{W^F}(I,\tau).$$
\end{lemma}
\begin{proof}
It is clear that $N_{W^F}(I) \subseteq N_{W^F}(I,\tau)$.
To show equality, it suffices to show that $N_{W^F}(I,\tau) \subseteq N_{W^F}(W_I)$.

So let $w \in N_{W^F}(I,\tau)$ and $\alpha \in I$.
Let $J \subseteq I$ be the $\tau$-orbit of 
$\alpha$ and $J' \subseteq I$ the $\tau$-orbit for which 
$w w_J w^{-1} = w_{J'}$.

We then have $$w_{J'} w \alpha = (w w_J w^{-1}) w \alpha = - w \alpha.$$
As the set of roots that are mapped to their inverses by $w_{J'}$ is precisely $\Phi_{J'}$, we have $w \alpha \in \Phi_{J'} \subseteq \Phi_I$.
As $\alpha \in I$ was arbitrary, we have $w \Phi_I = \Phi_I$ and consequently $w \in N_{W^F}(W_I)$.
\end{proof}
\begin{corollary}\label{cor: normalizers F-fixed and not}
We have $$N_{W^F}(W_I^F) =  N_{W^F}(W_I).$$
\end{corollary}

\begin{lemma}\label{la: normalizer of Levi and rational points} 
 We have $(N_{G}(\bf{T}) \cap N_{G}(L_I)) {L_I} = N_{G}(\bf{L}_I)$. 
\end{lemma}
\begin{proof}
 If $x \in N_{G}(\bf{L}_I)$, then $x \bf{T} x^{-1}$ is a maximally split torus of $\bf{L}_I$ and thus $L_I$-conjugate to $\bf{T}$. It follows that there exists $y \in L_I$ such that 
 $xy \in N_{G}(\bf{T})$. We clearly also have $xy \in N_{G}(L_I)$ and so $x = xy y^{-1} \in (N_{G}(\bf{T}) \cap N_{G}(L_I)) {L_I}$.
 
 Conversely, an element $x \in N_{G}(\bf{T}) \cap N_{G}(L_I)$
 defines an element $w$ of $W^F$. Moreover, we have $x L_I x^{-1} = L_I$ and thus $w \in N_{W^F}(W_I^F) =  N_{W^F}(W_I)$ by Corollary \ref{cor: normalizers F-fixed and not} and so $x \in N_{G}(\bf{L}_I)$.
\end{proof}

\begin{lemma}\label{la: Weyl group representative special choice}
 For every element $w \in N_{W}(I)$, there exists a representative $\dw \in N_{\bf{G}}(\bf{T}) \cap N_{\bf{G}}(\bf{L}_I)$ such that 
 $$ \dw u_\alpha(x) \dw^{-1} =  u_{w \alpha}(x)$$
 for all $\alpha \in I$ and $x \in \G_a$.
 
 In particular, the map $w \mapsto c_{\dw}$ defines an action of
 $N_{W}(I)$ on $\bf{L_I}$ which is compatible with the Steinberg morphism $F$ if $I$ is $\tau$-stable.
 
 If $I$ is $\tau$-stable and $Z(\bf{G})$ is connected, then the representatives $\dw$ for $w \in N_{W^F}(I)$ can be chosen in $N_{{G}}(\bf{T}) \cap N_{{G}}(\bf{L}_I)$.
 
\end{lemma}
\begin{proof}
 For every $w \in N_{W}(I)$, we start by taking any representative $\dw' \in N_{\bf{G}}(\bf{T}) \cap N_{\bf{G}}(\bf{L}_I)$ and consider the isotypy 
 $c_{\dw'} : \bf{L_I} \to \bf{L_I}$. 
 Its corresponding $p$-morphism of root data is essentially 
 given by 
 $X \to X, \chi \mapsto w^{-1}\chi$. 
 By \ref{remark: isotypies and pinnings}, we find $t \in \bf{T}$ such that $\dw = t \dw'$
 has the desired property.
 
 Let $w_1,w_2 \in N_W(I)$ and $w_3 = w_1 w_2$. Given representatives $\dw_i \in  N_{\bf{G}}(\bf{T}) \cap N_{\bf{G}}(\bf{L}_I)$ for $w_i $ ($i = 1,2,3$)
  with the desired property, we have $c_{\dw_1}\circ c_{\dw_{2}} = c_{\dw_3}$ as morphisms $\bf{L}_I \to \bf{L}_I$ by 
 Proposition \ref{prop: bijection isotypies p-morphisms}.
 
 The map $w \mapsto c_{\dw}$ thus defines an action of $N_{W}(I)$ on $\bf{L_I}$. 
 
 Suppose now that $I$ is $\tau$-stable.
 Now, let $w_1 \in N_W(I)$ and set $w_2 = F(w_1)$.
 Given representatives $\dw_1$ and $\dw_2$ with the desired property, the isotypy 
 $$ c_{F(\dw_1)} \circ F  = F \circ c_{\dw_{1}} $$
 is compatible with the given pinning.
 

 It follows from Remark \ref{remark: composition of compatible isotypies} that $c_{F(\dw_1)}$ is compatible with the pinning
 and Proposition \ref{prop: bijection isotypies p-morphisms} 
 now implies that $F(\dw)$ and $\dot{F(w)}$ differ by an element of $Z(\bf{L_I})$.
 It follows that the defined action is compatible with $F$.
 
 Suppose now also that $Z(\bf{G})$ is connected.
 For $w \in N_{W^F}(I)$ we can start with any representative $\dw' \in N_{{G}}(\bf{T}) \cap N_{{G}}(\bf{L}_I)$ and then Remark \ref{remark: rational isotypies and pinnings} implies the existence of some $t' \in T'$ such that $\dw = t' \dw'$ has the desired properties.
\end{proof}

\begin{corollary}\label{la: normalizer semidirect product}
 Suppose that $Z(\bf{G})$ is connected and that representatives  $\dw \in N_{G}(\bf{L_I}) \cap N_{G}(\bf{T})$ are chosen as in Lemma \ref{la: Weyl group representative special choice} for $w \in N_{W^F}(I)$.
 
We consider the group $${N} = \langle \dw \:|\:  w\in N_{W^F}(I) \rangle \rangle \subseteq N_{G}(\bf{L_I}) \cap N_{G}(\bf{T})$$ and its subgroup $Z = N \cap L_I$. We have $Z \subseteq Z(\bf{L}_I)$ and  
 the map $$N_{W^F}(I) \to N_G(L_I)/Z, w \mapsto \dw Z$$ is a morphism of groups making the short exact sequence 
 $$ 0 \to L_I/Z \to N_G(\bf{L_I})/Z \to  N_{W^F}(I) \to 0$$
 split.
\end{corollary}
\begin{proof}
 Every element of $N$ permutes the roots in $I$, but the only elements of $N_{L_I}(\bf{T})$ that do so lie in $T$,
 thus $N \cap L_I = N \cap T$. From the way the $\dw$ are chosen to act on $\bf{U}_\alpha$ ($\alpha \in I$) it follows that $N \cap T$ centralizes $\bf{U}_\alpha$ for all $\alpha \in I$.
 Now, Proposition \ref{prop: bijection isotypies p-morphisms}
 together with Lemma \ref{la: rational isotypies and pinnings}
 and Remark \ref{remark: rational isotypies and pinnings} implies that $N \cap T \subseteq Z(\bf{L_I})$ 
 
 We have canonical isomorphisms 
 $$N_{W^F}(I) \cong N_{G}(\bf{L_I})/L_I = N L_I/L_I \cong N/Z $$
 and the rest follows now immediately.
\end{proof}

\begin{remark}\label{remark: simply connected case}
 Suppose in addition to the assumptions of Lemma \ref{la: normalizer semidirect product} that $[\bf{G},\bf{G}]$
 is simply connected. Then the same is true for 
 $\bf{H}_I = [\bf{L}_I,\bf{L}_I]$ by  \cite[Prop. 12.14]{malle2011linear}
 and so 
 $\bf{H}_I$ is a direct product 
 $$ \bf{H}_I = \prod_{j} \bf{H}_{I_j} $$
 where $I = \bigcup_{j} I_j$ corresponds to the decomposition 
 $W_I^F = \prod_{j} W_{I_j}^F$ into irreducible components of the Coxeter group $W_I^F$.
 
 We then have the following diagram
 \[
  \xymatrix{
  N_{G}(\bf{L}_I) \ar@{->>}[d] & \\
  N_{G}(\bf{L}_I)/Z \ar[r]^{\sim} & L_I/Z \rtimes N_{W^F}(I) \\
  H_I \rtimes N_{W^F}(I) \ar@{->>}[r] &
  H_I Z/ Z \rtimes N_{W^F}(I) \ar@{^{(}->}[u] 
   \\
  \prod_{j} H_{I_j} \rtimes N_{W^F}(I) \ar[u]^{\sim} \ar@{->>}[r]
  &
  \prod_{j} H_{I_j} \rtimes N_{W^F}(I)/C 
  }
 \]
 where $C \subseteq N_{W^F}(I)$ is the centralizer of $I$.
 Note that the action of $N_{W^F}(I)$ on $\prod_{j} H_{I_j}$
 is induced by the action of $N_{W^F}(I)$ on $I$.
 
 Grouping those $I_j$ according to the types of 
 Coxeter groups $(W_{I_j},\tau)$ with automorphism 
 we find that $\prod_{j} H_{I_j} \rtimes N_{W^F}(I)/C$
 is isomorphic to a direct product of groups isomorphic to groups
 of the form $(H_{I_j} \rtimes A_{I_j}) \wr S_n$ where $A_{I_j}$ is the 
 group of diagram automorphisms of $\RR_{I_j}$.
 
 Note that since $A_{I_j}$ is isomorphic to a subgroup of $S_3$, it satisfies the assumption of Theorem \ref{thm: extending to wreath products}
 \end{remark}

 \section{Restriction along isotypies}
 \paragraph{}
In the following we consider an algebraically closed field
$k$ of characteristic different from $p$. Every connected 
reductive group $\bf{G}$ will be assumed to be equipped with a Steinberg morphism $F : G \to G$.

 \paragraph{}
Let $\varphi: \bf{G} \to \bf{G}'$ be an isotypy between connected reductive groups.
We shall write $\varphi: (\bf{G},F) \to (\bf{G}',F')$ to mean that $\varphi$ is compatible with the Steinberg morphisms, i.e.  
that $\varphi$ satisfies $F' \circ \varphi = \varphi \circ F$.

\paragraph{}
Recall that $\varphi :  (\bf{G},F) \to (\bf{G}',F')$ is called 
a regular embedding if $\varphi$ is an isomorphism onto its image and $Z(\bf{G}')$ is connected.
For the existence of regular embeddings, see \cite[Lemma 1.7.3]{geck2020character}.

\paragraph{}
Most of the properties of an isotypic morphism $\varphi$ can be deduced from its induced maps $\varphi$ between maximal tori
and thus from the induced maps between character and cocharacter 
groups.
\begin{lemma}\label{la: isotypic properties character cocharacter}
 Let $\varphi: (\bf{G},F) \to (\bf{G}',F')$ be an isotypic morphism, $\bf{T} \subseteq \bf{G}$ and $\bf{T}' \subseteq \bf{G}'$ maximally split tori such that $\varphi(\bf{T}) \subseteq \bf{T}'$. Denote with 
 $$ f_\varphi : X(\bf{T}') \to X(\bf{T}), \chi \mapsto \chi \circ \varphi $$
 the induced morphism between character groups and with 
 $$ f_\varphi^\vee: Y(\bf{T}) \to Y(\bf{T}'), \gamma \mapsto \varphi \circ \gamma $$
 the induced morphism between cocharacter groups.
 
 The following statements hold:
 \begin{enumerate}[(i)]
  \item The kernel of $\varphi$ is connected if and only 
        $\coker(f_\varphi)$ has no $p'$-torsion if and only if 
        $\coker(f_\varphi^\vee)$ has no $p'$-torsion.
  \item The kernel of $\varphi$ is finite if and only if 
        $\coker(f_\varphi)$ is finite.
  \item The morphism $\varphi$ is surjective if and only if 
        $\coker(f_\varphi^\vee)$ is finite.
  \item The morphism $\varphi$ is injective if and only if 
        $\coker(f_\varphi)$ is a finite $p$-group. 
 \end{enumerate}
\end{lemma}
\begin{proof}
 It is clear from the definition of isotypies that
 it suffices to consider the case $\bf{G} = \bf{T}$ and $\bf{G}'=\bf{T}'$.

 We consider the factorization 
 $$
 \xymatrix{\bf{T} \ar[r]^-\pi 
    & \bf{T}/\ker(\varphi)^\circ \ar[r]^{con_\varphi} 
    & \bf{T}/\ker(\varphi) \ar[r]^{can_\varphi} 
    & \varphi(\bf{T}) \ar[r]^-\iota
    & \bf{T}' }$$
 where $can_\varphi$ is the canonical map $\bf{T}/\ker(\varphi) \to \varphi(\bf{T})$ and $con_\varphi$ is the quotient map $$\bf{T}/\ker(\varphi)^\circ \to \bf{T}/\ker(\varphi.)$$ Here, $con$ stands for \emph{connected}, as $con_\varphi$ can be seen as a measure of how disconnected $\ker(\varphi)$ is.
 
 Applying $X(-)$, we see that 
 the torsion part of $\coker(f_\varphi)$ 
 is isomorphic to $\coker(f_{can_\varphi} \circ f_{con_\varphi})$
 which fits into an exact sequence 
 $$\xymatrix{ 0 \ar[r]& \coker(f_{can_\varphi}) \ar[r]& \coker(f_{can_\varphi} \circ f_{con_\varphi}) \ar[r]& \coker(f_{con_\varphi}) \ar[r]&0.} $$
 Since $\ker(\varphi)/\ker(\varphi)^\circ$ is a $p'$-group, so is
 $\coker(f_{con_\varphi}) = X(\ker(\varphi)/\ker(\varphi)^\circ)$ 
 and $\coker(f_{can_\varphi})$
 is a $p$-group by \cite[Prop. 3.8]{malle2011linear}.
 
 Using $Y(-)$ instead, one obtains an analogous 
 description of the torsion part of $\coker(f_\varphi^\vee)$
 (also see \cite[Prop 1.11]{bonnafe2006caracteres}). 
 
 It follows now 
 that no $p'$-torsion in either 
 $\coker(f_\varphi^\vee)$ or $\coker(f_\varphi)$ is equivalent 
 to $\ker(\varphi)$ being connected.
 
 The torsion-free part of $\coker(f_\varphi)$ is isomorphic
 to $\coker(f_\pi) = X(\ker(\varphi)^\circ)$ and is thus 
 trivial if and only if $\ker(\varphi)$ is finite. 
 
 Similarly, the torsion-free part of $\coker(f_\varphi^\vee)$ is
 isomorphic to $\coker(f_\iota^\vee) = Y(\bf{T}'/\varphi(\bf{T}))$ which is trivial if and only if $\varphi$ is surjective.

 We have thus obtained (i), (ii) and (iii) and (iv) is just 
 a combination of (i) and (iii).
\end{proof}

\paragraph{}
The concept of duality (see \cite[1.5.17]{geck2020character}) plays a very important role in the context of isotypies. 
From the description via $p$-morphisms of root data and the Existence Theorem \cite[Theorem 3.8]{taylor2019rootdata}, it follows
that every isotypy $\varphi : G \to G'$
admits a dual morphism $\varphi^*: G'^* \to G^*$
such that $f_\varphi$ corresponds to $f_{\varphi^*}^\vee$ under duality and likewise $f_{\varphi}^\vee$ to $f_{\varphi^*}$.
In accordance with Proposition \ref{prop: bijection isotypies p-morphisms},
a dual morphism is only unique up to conjugation.
\begin{corollary}\label{cor: duality and properties of isotypies}
 Let $\varphi: (G,F) \to (G',F')$ be an isotypy and 
 let $ G^*$ and  $G'^*$ be groups in duality with $G$ and $G'$,
 respectively. Let $F^* : G^* \to G^*$ and $F'^*:G'^* \to G'^*$
 be Steinberg morphisms dual to $F$ and $F'$, respectively.
 Then there exists a dual morphism $\varphi^*: (G'^*,F'^*) \to (G^*,F^*)$ of $\varphi$ and 
 the following statements hold:
 \begin{enumerate}
  \item The kernel of $\varphi$ is connected if and only if the kernel of $\varphi^*$ is connected.
  \item $\varphi$ is surjective if and only if $\ker(\varphi^*)$ is finite.
  \item $\varphi$ is injective if and only if $\varphi^*$ is surjective and $\ker(\varphi^*)$ is connected. 
 \end{enumerate}
\end{corollary}
\begin{proof}
 Once we have established the existence of $\varphi^*$,
 Lemma
 \ref{la: isotypic properties character cocharacter}
 gives us the remaining claims.
 
 If $\psi : G'^* \to G^*$ is any dual morphism of $\varphi$,
 then $F' \circ \varphi = \varphi \circ F$ implies 
 $c_{t'} \circ \psi \circ F'^* = F^* \circ \psi$ for some 
 $t' \in T^*$ where $T^* \subseteq G^*$ is the dual torus of $T$.
 
 By \cite[1.4.8]{geck2020character}, there exists $t \in T^*$
 such that $t' = F^*(t^{-1})t$, so that $\varphi^* = c_{t} \circ \psi: (G'^*,F'^*) \to (G^*,F^*)$ is a dual morphism of $\varphi$. 
\end{proof}

\paragraph{}
If $\varphi: (\bf{G},F) \to (\bf{G}',F')$ is an isotypy, we have an induced morphism 
$ \varphi: G \to G' $ and so for every field $k$, we have 
induction and restriction functors 
$$ \Res_\varphi: k G'\Mod \to k G\Mod $$
and 
$$ \Ind_\varphi: k G\Mod \to k G'\Mod $$
which are biadjoint.

\paragraph{}
Note that $\Res_\varphi$ can be viewed as the composition
of restriction from $G'$ to $\varphi(G)$ with inflation from
$\varphi(G) \cong G/\ker(\varphi)^F$ to $G$.
In particular, $\Res_\varphi$ maps semisimple modules to semisimple ones.

\paragraph{}
If $\varphi' :  (\bf{G}',F') \to (\bf{G}'',F'')$ is another 
isotypy, then we have transitivity
$$ \Res_{\varphi' \circ \varphi} = \Res_{\varphi} \circ \Res_{\varphi'}$$
of restriction and similarly for induction.
%

\paragraph{}
The following result is a generalization of \cite[1.7.15]{geck2020character} from regular embeddings to arbitrary isotypies.
We shall abbreviate its conclusion by saying that $\Res_\varphi$ is multiplicity-free.
\begin{theorem}\label{thm: multiplicity-free}
 Let $\varphi: (\bf{G},F) \to (\bf{G}',F') $ be an isotypy and $X$ a simple $k G'$-module. Then $\Res_\varphi X$
 is multiplicity-free.
\end{theorem}
\begin{proof}
  Let $\bf{T} \subseteq \bf{G}$ be a maximally split maximal torus and let $\bf{T}' \subseteq \bf{G}'$ be the maximally split torus
 of $\bf{G}'$ given by $\bf{T}' = Z(\bf{G}')^\circ \varphi(\bf{T})$.
 
 Let $(\bf{G}^*,F^*)$ and $(\bf{G}'^*,F'^*)$ be groups in duality 
 with $(\bf{G},F)$ and $(\bf{G}',F')$ with $\bf{T}^* \subseteq \bf{G}^*$ and $\bf{T}'^* \subseteq \bf{G}'^*$ the maximally split tori  corresponding to $\bf{T}$ and $\bf{T}'$, respectively.
 
 By \ref{cor: duality and properties of isotypies}, there exists a dual morphism $\varphi^* : (\bf{G}'^*,F'^*) \to (\bf{G}^*,F^*)$ of $\varphi$. 
 
 We consider the factorization
 \[
  \xymatrix{  
   \bf{T}'^* \times_{\ker(\varphi^*)} \bf{G}'^* \ar[rr] && \varphi^*(\bf{G}') \ar[d] \\
   \bf{G}'^* \ar[rr]^{\varphi^*} \ar[u]   && \bf{G}^*
  }
 \]
 where the morphisms are given by
 $$\bf{G}'^* \to \bf{T}^* \times_{\ker(\varphi^*)} \bf{G}'^*, g' \mapsto (1,g') $$
 and
 $$\bf{T}'^* \times_{\ker(\varphi^*)} \bf{G}'^* \to \varphi(\bf{G}'^*), (t',g') \mapsto \varphi^*(g')$$
 as well as the inclusion $\varphi^*(\bf{G}'^*) \to \bf{G}^*$.
 
 All three of these are isotypies compatible with the 
 naturally induced Steinberg morphisms.
 
 Note that the first and third is an isomorphism onto its image
 while the second one is a morphism with connected kernel.
 
 By duality, we obtain a factorization 
 $$
 \xymatrix{  
   \bf{G}_1 \ar[rr]^{\psi} && \bf{G}_2 \ar[d]^{\psi_2} \\
   \bf{G} \ar[rr]^{\varphi} \ar[u]^{\psi_1}   && \bf{G'}
  }
 $$
 where 
 $\psi_1: (\bf{G},F) \to (\bf{G}_1,F_2)$ and $\psi_2: (\bf{G}_2,F_2) \to (\bf{G}',F')$
 are surjective with connected kernel and $\psi: (\bf{G}_1,F_1) \to (\bf{G}_2,F_2)$ is injective by Corollary (\ref{cor: duality and properties of isotypies}).
 
 In particular, the functors $\Res_{\psi_1}$ and $\Res_{\psi_2}$ are just inflations and map simple objects to simple ones.
 
 It thus suffices to show that $\Res_\varphi$ is multiplicity-free, when $\varphi$ is injective. In this case,
 $\varphi : G \to \varphi(G) = \varphi(\bf{G})^F$ is an isomorphism, so that we may as well replace $\bf{G}$ by $\varphi(\bf{G})$ and thus assume that $\varphi$ is an isomorphism
 onto its image.
 
 Now, let $\iota: (\bf{G}',F') \to (\bf{G}'',F'')$ be a regular embedding and  note that $\iota \circ \varphi$ is then a regular embedding, too.
 By \cite[1.7.15]{geck2020character}, we know that $\Res_{\iota \circ \varphi} X'$ is multiplicity-free for every simple $k G''$-module $X'$. As $X'$ runs through the isomorphism classes of simple $k G''$-modules, every simple $k G'$-module $X$ occurs as a direct summand in some $\Res_\iota X'$.
 It follows that $\Res_{\varphi} X$ is a summand of $\Res_\varphi \Res_{\iota} X' \cong \Res_{\iota \circ \varphi} X'$ and thus multiplicity-free.
 \end{proof}

\section{Unipotent representations}
\paragraph{}
From now on let $(K,R,k)$ denote an $\ell$-modular system
which is split for all finite groups under consideration.
In particular, all the results from previous sections (where $k$ was assumed to be algebraically closed) descent to this situation. We let $\Lambda \in \{K,k\}$.

\paragraph{}
Let $(\bf{G},F)$ be a connected reductive group and $F : \bf{G} \to \bf{G}$ a Steinberg morphism.
We have a decomposition 
$$\Irr_K(G) = \bigsqcup_{[s]} \EE(G,[s])$$
of the set of irreducible characters into Lusztig series 
(cf. \cite[2.6.2]{geck2020character}) as well as a decomposition 
$$\IBr_\ell(G) = \bigsqcup_{[s]} \wh{\EE}_\ell(G,[s])$$ 
of the set of irreducible Brauer characters into $\ell$-modular
Lusztig series (cf. \cite[9.12]{cabanes2004representation}).
Here, the first union is over $G^*$-conjugacy classes of $F^*$-stable semisimple elements in 
a group $(\bf{G}^*,F^*)$ in duality with $(\bf{G},F)$.
The second union is over $G^*$-conjugacy classes of $F^*$-stable semisimple $\ell'$-elements.

\paragraph{}
Representations and modules admitting characters in $\EE(G,[1])$
or $\wh{\EE}_\ell(G,[1])$ are called \emph{unipotent}.

\paragraph{}\label{par: Lusztig series reduction modulo ell}
The two decompositions are compatible with reduction modulo $\ell$. In fact, if $\chi \in \EE(G,[s])$ and $\psi \in \IBr_\ell(G)$ is a constituent of the reduction $\wh{\chi}$ of 
$\chi$ modulo $\ell$, then $\psi \in \wh{\EE}_\ell(G,[s_{\ell'}])$
where $s_{\ell'}$ denotes the $\ell'$-part of $s$.

\paragraph{}\label{par: Lusztig series linear characters}
Moreover, the decomposition into Lusztig series is compatible with the action of the group of linear characters.
In fact, every linear character $\lambda: G \to K^*$ whose kernel contains  
every $F$-stable unipotent element corresponds to a unique 
$z_\lambda \in Z(G^*)$ (cf. \cite[2.5.20]{geck2020character}) with the same order as $\lambda$
and one has bijections 
$$ \EE(G,[s]) \to \EE(G,[zs]), \chi \mapsto \lambda \chi$$
and
$$ \wh{\EE}_\ell(G,[s_{\ell'}]) \to \wh{\EE}_\ell(G,[z_{\ell'}s_{\ell'}]), \psi \mapsto \wh{\lambda} \psi$$
where the first one comes from \cite[2.5.21]{geck2020character}
and the second is a consequence of the first and 
\ref{par: Lusztig series reduction modulo ell}
%

\begin{lemma}\label{la: isotypies unipotent remains simple}
 Let $\varphi: (\bf{G},F) \to (\bf{G}',F')$ be an isotypy.
 For every simple unipotent $\Lambda G'$-module $X'$, its restriction $\Res_\varphi X'$ is a simple unipotent $\Lambda G$-module and every simple unipotent $\Lambda G$-module $X$
 is of the form $\Res_\varphi X'$ for a simple unipotent $\Lambda G'$-module $X'$ which is unique up to isomorphism.

 In particular, $Z(G)$ acts trivially 
 on every simple unipotent $\Lambda G$-module.
\end{lemma}

\begin{proof}
 These statements are well-known for the case $\Lambda = K$
 (see \cite[2.3.14]{geck2020character}).
 We already know from \ref{thm: multiplicity-free}, the 
 $\Res_\varphi X'$ is multiplicity-free. 
 Since
 $\varphi(G)$ contains all $F'$-stable unipotent
 elements of $G'$, the kernel of 
 every $\lambda \in \Irr_\Lambda(G'/\varphi(G))$ contains  
 all $F'$-stable unipotent
 elements of $G'$.
 It follows from \ref{par: Lusztig series linear characters} that $\lambda \otimes X'$ is not unipotent
 unless $\lambda$ is trivial.
 
 It follows that \ref{la: abelian multiplicity-free}
 applies and so $\Res^{G'}_{\varphi(G)} X'$ is simple.
 Since $\Res_\varphi X'$ is just the inflation of
 $\Res^{G'}_{\varphi(G)} X'$ to $G$ along $\varphi$,
 the first claim follows.
 
 Furthermore, this also implies that distinct simple unipotent $\Lambda G'$-modules restrict to distinct 
 $\Lambda \varphi(G)$-modules, because $\Res_\varphi X' \cong \Res_\varphi X''$ implies that we have $X'' \cong \lambda \otimes X'$ for some $\lambda \in \Irr_\Lambda(G'/\varphi(G))$ 
 (cf. \cite[Section 2.8]{kowalski2014introduction}). 
 
 It remains to show that every simple unipotent 
 $\Lambda G$-module is of the form $\Res_\varphi X'$
 for a simple unipotent $\Lambda G'$-module $X'$.
 This follows essentially from \cite[3.3.24]{geck2020character} together with our previous 
 remarks in \ref{par: Lusztig series reduction modulo ell}.
 Note that the assumption of a connected kernel in 
 \cite[3.3.24]{geck2020character} is redundant 
 as one can write any isotypy as a composition of isotypies
 which have connected kernels (see, for instance, the proof of Theorem \ref{thm: multiplicity-free}).
 
 The last claim about the center $Z(G)$ acting trivially 
 on simple unipotent representations follows by 
 considering the isotypy $\bf{G} \to \bf{G}/Z(\bf{G})$.
\end{proof}
\paragraph{}
Finally, we have all the tools to obtain our main result. 
\begin{theorem}
 Let $\bf{L} \subseteq \bf{G}$ be a split Levi subgroup of $\bf{G}$, let $X$ be 
 a simple unipotent $\Lambda L$-module. 
 
 Then $X$ extends to a module over the stabilizer $N_{G}(\bf{L},X)$ of $X$ inside the normalizer 
 $N_{G}(\bf{L})$.
\end{theorem}
\begin{proof}
 First, let $\iota : (\bf{G},F) \to (\bf{G}',F')$ be a regular
 embedding and set $\bf{L}' = Z(\bf{G}')\varphi(\bf{L})$ which is a
 split Levi subgroup of $\bf{G}'$.
 
 According to Lemma \ref{la: isotypies unipotent remains simple} there exists a simple unipotent $\Lambda L'$-module $X'$ 
 such that $\Res_\iota X' = X$.
 
 Clearly, we have $\iota(N_{G}(\bf{L},X)) \subseteq N_{G'}(\bf{L}',X')$ and thus if $X'$ extends to a module over 
 $N_{G'}(\bf{L}',X')$, then $X$ extends to a module over 
 $N_{G}(\bf{L},X)$.
 
 We thus may assume that $Z(\bf{G})$ is connected. According to \cite[1.7.13]{geck2020character}
 there exists a surjective isotypy $\varphi: (\bf{G}_0,F_0) \to (\bf{G},F)$
 such that $Z(\bf{G}_0)$ and $\ker(\varphi)$  are connected
 and that the derived subgroup $[\bf{G}_0,\bf{G}_0]$ is simply-connected.
 
 We set $\bf{L}_0 = \varphi^{-1}(\bf{L})$ as well as $X_0 = \Res_\varphi X$ which is just the inflation of $X$ to $L_0$. 
 We have 
$\varphi(N_{G_0}(\bf{L}_0,X_0)) = N_{G}(\bf{L},X)$ and if $X_0$
extends to a $\Lambda N_{G_0}(\bf{L}_0,X_0)$-module, then
$\ker(\varphi)^{F_0}$ acts trivially on this module (as it acts trivially on $X_0$) and thus induces a $\Lambda N_{G}(\bf{L},X)$-module which extends $X$.

We may thus also assume that the derived subgroup of $\bf{G}$
is simply connected. Moreover, we may assume that $\bf{L} = \bf{L}_I$ is a standard Levi subgroup so that we are in the situation 
described in Remark \ref{remark: simply connected case} and 
we shall make use of the notation which was used there. 

The center $Z(\bf{L})$ acts trivially on $X$ by Lemma \ref{la: isotypies unipotent remains simple} and its restriction $X_0 = \Res^{L_I}_{H_I} X$ to $H_I$
is simple by the same Lemma because the inclusion $\bf{H}_I \to \bf{L}_I$ is an isotypy. In particular, we can consider $X$ as a $\Lambda L_I/Z$-module and it suffices to show that 
$X$ extends to a $N_{G}(\bf{L}_I,X)/Z$-module.

Theorem \ref{thm: extending to wreath products} implies that $X_0$ extends to a $\Lambda (H_I \rtimes N_{W}(I,X))$-module.  
Finally, as we have 
$$\Hom_{\Lambda L}(\ad(w)X,X) \cong \Hom_{\Lambda H_I}(\ad(w)X_0,X_0),$$
the action of $N_{W}(I,X)$ on $X_0$ is compatible with the action of $L_I$ on $X_0$. Thus the extension of $X_0$
to $H_I \rtimes N_{W}(I,X)$ can be considered as an extension
of $X$ as a $\Lambda (L_I/Z \rtimes N_{W}(I,X))$-module which is the same as a $\Lambda N_{G}(\bf{L}_I,X)/Z$-module
by Remark \ref{remark: simply connected case} and so the claim follows.
\end{proof}

\bibliographystyle{abbrv}

\end{document}